\documentclass{amsart}
\usepackage{amssymb,latexsym}
\usepackage{amsmath}
\usepackage[dvipdfm]{graphicx}
 \usepackage{color}
\usepackage{enumerate}
\newenvironment{enumeratei}{\begin{enumerate}[\upshape (i)]}%
                            {\end{enumerate}}
\newcommand \url [1] {\texttt{#1}}
\numberwithin{equation}{section}
\theoremstyle{plain}
 \newtheorem{theorem}{Theorem}[section]
 
 \newtheorem{lemma}[theorem]{Lemma}
 
 \newtheorem{corollary}[theorem]{Corollary}
\theoremstyle{definition}

\theoremstyle{remark}
 
\newcommand \dlat[1] {\pmb{DRL}(#1)}
\newcommand \rlat[1] {\pmb{RL}(#1)}
\newcommand \boole [1] {B_{#1}}
\newcommand \algboole [1] {\alg B_{#1}}
\newcommand \atom [2] {a^{(#1)}_{#2}}
\newcommand \btom [2] {b^{(#1)}_{#2}}
\newcommand \orb [1]  {\textup{Orb}(#1)}
\newcommand \gen [1] {[#1]_{\textup{RotLat}}}
\newcommand \latgen [1] {[#1]_{\textup{Lat}}}
\newcommand \oid {\mathcal I_{\text{fin}}(\mathbb N)}
\newcommand \vvv [1] {\pmb{V\kern-1.8pt ar}(#1)}
\newcommand \var [1] {{\mathcal #1}}
\newcommand \alg [1] {\mathfrak #1}
\renewcommand \emptyset{\varnothing}
\newcommand \tuple [1] {\langle #1 \rangle}
\newcommand \tbf[1] {\textbf{#1}}       
\renewcommand\rho{\varrho}
\renewcommand\phi{\varphi}
\newcommand \id {\textup{id}}
\newcommand \pair [2] {\tuple{ #1,#2}}

\newcommand \set[1] {\{#1\}}
\newcommand \restrict [2] {{#1}\kern-1pt \rceil_{\kern-1pt #2}}
\newcommand \length [1] {\textup{length}\,#1}
\newcommand \bigset[1] {\bigl\{#1\bigr\}}
\newcommand\nothing [1] {}

\begin{document}
\title[Varieties of distributive rotational lattices]
{Varieties of distributive rotational lattices}
\author[G.\ Cz\'edli]{G\'abor Cz\'edli}
\email{czedli@math.u-szeged.hu}
\urladdr{http://www.math.u-szeged.hu/$\sim$czedli/}
\address{University of Szeged\\Bolyai Institute\\
Szeged, Aradi v\'ertan\'uk tere 1\\HUNGARY 6720}

\author[I.\,V.\ Nagy]{Ildik\'o V.\ Nagy}
\email{ildi.ildiko.nagy@gmail.com}

\thanks{This research was supported by the NFSR of Hungary (OTKA), grant numbers  K77432 and
K83219, and by  T\'AMOP-4.2.1/B-09/1/KONV-2010-0005}


\nothing{06B20 (1980-now) Varieties of lattices } 
\nothing{06B75 (2010-now) Generalizations of lattices }
\nothing{06D99 (1980-now) Distributive lattices: None of the above, but in this section }


\date{September 14, 2012, revised April 23, 2013}

\begin{abstract} 
A rotational lattice is a structure $\tuple{L;\vee,\wedge, g}$ where $L=\tuple{L;\vee,\wedge}$ is a lattice and $g$ is a lattice automorphism of finite order. 
We describe the subdirectly irreducible distributive rotational lattices. Using J\'onsson's lemma, this leads to a description of all varieties of distributive rotational lattices. 
\end{abstract}

\maketitle

\section{Introduction and target}
A \emph{rotational lattice} is a structure $\alg L=\tuple{L;\vee,\wedge, g}$ where $L=\tuple{L;\vee,\wedge}$ is a lattice, $g$ is an automorphism of this lattice, and $g^n$ equals the identity map $\id_L$ on $L$ for some $n\in\mathbb N=\set{1,2,3,\dots}$. The smallest $n\in\mathbb N$ such that $g^n=\id_L$, that is the identity $g^n(x)=g(g(\dots g(x)\dots))\approx x$ (with $n$ copies of $g$) holds in $\alg  L$, is the \emph{order} of $\alg L$. If the lattice reduct $\tuple{ L;\vee,\wedge}$ of $\alg L$ is distributive, then $\alg L$ is a \emph{distributive rotational lattice}.
For $n\in\mathbb N$, let $\rlat n$ denote the class of  rotational lattices satisfying the identity $g^n(x)\approx x$, and let $\dlat n$
be the class of distributive members of $\rlat n$. 

The concept of rotational lattices was introduced by Chajda, Cz\'edli and Hala\v s \cite{chajdaczedlihalas}. The members of $\rlat 2$ are called \emph{lattices with involution}, and they were studied in several papers, including Chajda and Cz\'edli~\cite{chczinv}. Distributive involution lattices play the main role in understanding the compatible quasiorderings of lattices in  Cz\'edli and Szab\'o~\cite{czedliszabo}. Boolean rotational lattices and even more general structures are interesting in \L ukasiewicz logic, see Vetterlein~\cite{vetterlein}.  
The study of rotational lattices and the present work are also motivated by Je\v zek~\cite{jezek} and  Mar\'oti~\cite{maroti}, who described the simple and the subdirectly irreducible rotational semilattices,  by Dziobiak, Je\v zek, and Mar\'oti, who determined the minimal quasivarieties of rotational semilattices, and by Nagy~\cite{ildiko}, who went even further. 

Although semilattices constitute a minimal variety, Dziobiak, Je\v zek, and Mar\'oti \cite{dziobiakatal}, and 
the above-men\-tioned papers,   \cite{jezek}, \cite{maroti}, and \cite{ildiko}, witness that their rotational variants are quite complicated. This is why the present paper is restricted to the distributive case. If distributivity is disregarded, then even $\rlat 1$, which is equivalent to the class of all lattices, becomes quite complicated.

\subsection*{Target} The class of all distributive rotational lattices is not a variety since it is clearly not closed under taking direct products. However, this class includes some varieties, like $\dlat n$ for $n\in\mathbb N$. After describing the subdirectly irreducible distributive rotational lattices, we 
also describe the varieties of distributive rotational lattices. There are countably many of these varieties, and many of them are not of the form  $\dlat n$.

\section{The result}
Let $\boole n=\tuple{\boole n;\vee,\wedge}$ denote the boolean lattice of length $n$, that is of size $2^n$. Let $\atom n0,\dots,\atom n{n-1}$ be its atoms. To define an automorphism $g$ of $\boole n$, it suffices to give the action of $g$ on the set of atoms. Let $g(\atom ni)=\atom n{i+1}$ where $i+1$ is understood modulo $n$. This way we obtain the \emph{$n$-dimensional rotational cube} $\algboole n=\tuple{\boole n;\vee,\wedge, g }$. Its order is $n$.  
The  divisibility relation on $\mathbb N=\set{1,2,3,\dots}$ is denoted in the usual way: $a\mid b$ if $b=ac$ for some $c\in\mathbb N$. The set of finite order ideals of the poset $\tuple{\mathbb N;\mid\,}$ will be denoted by $\oid$; a subset $X$ of $\mathbb N$ belongs to $\oid$ if{f} $X$ is finite and, for all $x,y\in \mathbb N$, $x\in X$ and $y\mid x$ imply $y\in X$. For $X\in \oid$, let $\vvv X$ denote the variety generated by $\set{\algboole n: n\in X}$. 
Note that  $\vvv\varnothing$ consists of singleton algebras.
Now we are in the position to formulate our result.

\begin{theorem}\label{thmmain}\ 
\begin{enumeratei}
\item\label{thmmaina} The subdirectly irreducible distributive rotational lattices are exactly the rotational cubes $\algboole n$, $n\in\mathbb N$. These $\algboole n$ are simple.
\item\label{thmmainb} The varieties of distributive rotational lattices are exactly the $\vvv X$, $X\in \oid$. For $X,Y\in\oid$,
we have $\vvv X\subseteq \vvv Y$ if{f} $X\subseteq Y$. 
\item\label{thmmainc} For $n\in\mathbb N$, 
 $\dlat n=\vvv{\set{x:x\text{ divides } n} }$.
\end{enumeratei}
\end{theorem}

\section{Auxiliary statements and proofs}
Rotational lattices are often denoted by Fraktur letters like $\alg A$, $\alg B$, $\alg D$, $\alg L$, and $\alg M$; the corresponding italic letters, 
$A$, $B$, $D$, $L$, and $M$, will stand for 
their lattice reducts and base sets.
An element $a$ of a rotational lattice $\alg L=\tuple{L;\vee,\wedge, g}$
 is  \emph{stable} if $g(a)=a$. In the following lemma, we do not assume $0,1\in L$.

\begin{lemma}\label{stablemma} Let $\alg L$ be a subdirectly irreducible distributive rotational lattice. If $a\in L$ is a stable element, then $a$ is either the least element $0=0_L$ of $L$, or the greatest element $1=1_L$ of $L$.
\end{lemma}

\begin{proof} Our argument is motivated by Gr\"atzer~\cite[Example 218]{GGLT}. For the sake of contradiction, suppose $a\in  L$ is stable but $a$ is neither the smallest, nor the largest element of $L$. Define $\alpha=\set{\pair xy\in L^2: a\vee x=a\vee y}$ and its dual, 
$\beta=\set{\pair xy\in L^2: a\wedge x=a\wedge y}$. 

It belongs to the folklore that $\alpha$ and $\beta$ are lattice congruences; we mention only one step from the argument:  if $\pair{x_i}{y_i}\in\alpha$ for $i\in\set{1,2}$, then 
\[  (x_1\wedge x_2)\vee a=  (x_1\vee a)\wedge (x_2\vee a) = (y_1\vee a)\wedge (y_2\vee a) = (y_1\wedge y_2)\vee a
\]
shows that $\pair{x_1\wedge x_2}{y_1\wedge y_2}\in\alpha$. If $\pair xy\in\alpha$, then 
\[g(x)\vee a=g(x)\vee g(a)=g(x\vee a)=g(y\vee a)=g(y)\vee g(a)= g(y)\vee a
\]
yields $\pair{g(x)}{g(y)}\in\alpha$. Hence $\alpha$ is a congruence of $\alg L$, and so is $\beta$ by duality.

Since $a\neq 0_L$, there is a $b\in L$ such that $b < a$, and $\pair ab\in \alpha$ shows that $\alpha$ is distinct from $\omega_{\alg L}$, the smallest congruence on $\alg L$. The dual consideration shows $\beta\neq \omega_{\alg L}$.
However, $\alpha\cap\beta=\omega_{\alg L}$ by the cancellativity rule of distributive lattices, see Gr\"atzer~\cite[Corollary 103]{GGLT}. This is a contradiction since the subdirect irreducibility of $\alg L$ implies that  $\omega_{\alg L}$ is completely meet-irreducible in the lattice of congruences of $\alg L$, see Burris and Sankappanavar~\cite[Theorem 8.4.]{burrissankap}.
\end{proof}

A subalgebra $\alg M$ of $\alg L$ is a \emph{spanning subalgebra} if $0_{M}=0_L$ and $1_{M}=1_L$.

\begin{corollary}\label{spancor} Let $\alg M$ be a subalgebra of a subdirectly irreducible distributive rotational lattice $\alg L$ such that $g$, restricted to $M$, is not the identity map of $M$. Then $\alg M$ is a spanning subalgebra of $\alg L$.
\end{corollary}

\begin{proof} Assume $g(a)\neq a\in M$, and let $n$ be the order of $\alg L$. Then $\bigvee\set{g^i(a):0\leq i<n}\in M$ is a stable element, and it is greater than $a$. Hence this join is $1_L$ by Lemma~\ref{stablemma}. The dual argument shows $0_L\in M$. 
\end{proof}

An algebra is \emph{locally finite} if each of its finite subsets generates a finite subalgebra.

\begin{lemma}\label{locfinlemma} Let $t$ be a $k$-ary term in the language of rotational lattices. Then, for each  $n\in \mathbb N$, there exists a $kn$-ary lattice term $p_n$ such that  identity
\begin{equation*}
t(x_1,\dots,x_k)\approx p_n(x_1,g(x_1),\dots,g^{n-1}(x_1),\dots, x_k,g(x_k),\dots,g^{n-1}(x_k))
\end{equation*}
holds in all rotational lattices of order $n$.
Consequently, every distributive rotational lattice is locally finite.
\end{lemma}

\begin{proof} Since $g$ commutes with lattice terms, a  straightforward induction yields the first part of the statement. The second part follows from the fact that distributive lattices are locally finite.   
\end{proof}

The following lemma belongs to the folklore.

\begin{lemma}\label{atomBool} Let $a_1,\ldots,a_t$ be distinct atoms of a distributive  lattice $D$. Then the sublattice generated by $\set{a_1,\ldots,a_t}$ is $($isomorphic to$)$ the $2^t$-element boolean lattice.
\end{lemma}

\begin{proof} We obtain $(a_1\vee\dots \vee a_{i-1})\wedge a_i=0$ from distributivity. Thus Gr\"atzer~\cite[Theorem 360]{GGLT} applies.
\end{proof}

For $a\in  L$, the \emph{orbit} of $a$ is $\orb a=\set{g^i(a): i\in\mathbb N_0}$. It is a finite subset of $L$. Note that $a$ is stable if{f} $|\orb a|=1$. If $\alg M$ is a subalgebra of $\alg L$, then the restriction of $g$ to $M$ will be denoted by $\restrict gM$. It may happen that $g$ and $\restrict gM$ are of different orders as permutations; that is, $(\restrict gM)^k=\id_M$ does not imply $g^k=\id_L$ in general.

\begin{lemma}\label{eoJsHGs}  Let $\alg L$ be a subdirectly irreducible distributive rotational lattice, and let $a\in  L$ be a non-stable element. Then, denoting $|\orb a|$ by $n$,  the subalgebra $\gen a$  of $\alg L$ generated by $\set a$  is $($isomorphic to$)$ the $n$-dimensional rotational cube $\algboole n$.
\end{lemma}

\begin{proof} Let $\alg A=\gen { a}$, the subalgebra $\tuple{A;\vee,\wedge,g}$ generated by $\set a$. It follows from Lemma~\ref{locfinlemma} that 
$\tuple{A;\vee,\wedge}$ is generated by $\orb a$; in notation, $A=\latgen{\orb a}$. It also follows that $\tuple{A;\vee,\wedge}$  is a  finite distributive lattice, and we know from Corollary~\ref{spancor} that it is a spanning sublattice of $\tuple{L;\vee,\wedge}$. 
Since  $(\restrict gA)^n$ acts identically on  the generating set $\orb a$ of $\tuple{A;\vee,\wedge}$, we obtain that  $(\restrict gA)^n=\id_A$ and $\alg A$ is of order $n$.
Pick an atom $b$ of $A$ such that $b\leq a$. It is not stable by Lemma~\ref{stablemma} and Corollary~\ref{spancor}. Hence, denoting $|\orb b|$ by $m$, we have $1<m$. 
We know from Lemma~\ref{atomBool} that 
$\set{g^i(b): 0\leq i< m}$ generates a boolean sublattice $B$ of length $m$ in the lattice $L$. Since $b\in A$, \ $B$ is also a sublattice of $A$. Obviously, $\alg B=\tuple{B;\vee,\wedge,g}$ is the $m$-dimensional rotational cube $\algboole m$.

Clearly, $1_B=\bigvee\set{g^i(b):  0\leq i< m}$ is a stable element in $\alg L$. Hence,  
applying  Lemma~\ref{stablemma} to $\alg L$ and Corollary~\ref{spancor}, we obtain $1_B=1_L=1_A$ and, of course, $0_B=0_A$. 
Next, to show $\length A=\length B$, take a maximal chain $C$ in $B$, and let $u\prec_B v$ be two consecutive members of this chain. Denote by $w$ the unique complement of $u$ in the interval $[0,v]$ of $B$. Since $[u,v]$ is down-perspective to $[0,w]$ in $B$, we obtain that $w$ is an atom of $B$. Hence $w=g^i(b)$ for some $i\in\set{0,\ldots, m-1}$. Since $g$ sends atoms to atoms, $w$ is also an atom of $A$. Hence, the above-mentioned perspectivity yields that $v$ covers $u$  in $A$. Therefore, $C$ is a maximal chain of $A$, and $\length A=\length B=m$.
Since each distributive lattice with length $m$ has at most $2^m$ elements by \cite[Corollary 112]{GGLT},  and since $B\subseteq A$, we conclude that $B=A$. 
\end{proof}

The following statement is almost obvious.

\begin{lemma}\label{bolaheigts} If $B$ is a spanning boolean sublattice of a finite distributive lattice $L$, then 
$\length L=\sum\{h(a): a\textup{ is an atom of }B\}$, where $h(a)$ denotes the height of $a$.
\end{lemma}

\begin{proof} Let $a_0,\dots,a_{n-1}$ be the atoms of $B$. For $i=0,\dots,n-1$, the length of the interval $[a_0\vee\dots \vee a_{i-1},a_0\vee\dots\vee a_{i}]$ is $h(a_{i})$, because this interval is perspective to the interval $[0,a_{i}]$. Extending $\set{0,a_0,a_0\vee a_1,\dots, a_0\vee\dots\vee a_{n-1}=1}$ to a maximal chain of $L$, the statement follows.
\end{proof}

If $\alg A$ is a subalgebra of $\alg L$ such that each covering pair of elements within $\alg A$ is a covering pair in $\alg L$, then $\alg A$ is a \emph{cover-preserving subalgebra}.

\begin{lemma}\label{odCkGBm}  Let $\alg L$ be a subdirectly irreducible distributive rotational lattice, and  let $n=\max\set{|\orb w|: w\in L}$. Then $\alg L$  is $($isomorphic to$)$ the $n$-dimensional rotational cube $\algboole n$.
\end{lemma}

\begin{proof} We   assume $n\geq 2$ since otherwise the statement is well-known; see Gr\"atzer~\cite[Example 218]{GGLT}. 
Pick an element $w\in L$ such that $n=|\orb w|$. We know from Lemma~\ref{eoJsHGs} that the subalgebra $\alg A=\gen w$ is the $n$-dimensional rotational cube $\algboole n$. For the sake of contradiction, suppose $A\neq L$. If we had $\length L\leq n$, then $|L|\leq 2^n=|\algboole n|=|A|$ would  give $L=A$, a contradiction. Thus $\length L > n$, and the spanning subalgebra $\alg A$ is not a cover-preserving subalgebra. Hence there is 
a prime interval $[u,v]$, that is a covering pair $u\prec_A v$,  of $A$ such that $v$ does not cover $u$ in $L$. 
Let $a$ be the (unique) relative complement of $u$ in $[0,v]$, understood within $A$. Then $a$ is an atom of $A$, and 
$[u,v]$ is perspective to $[0,a]$. Since $[u,v]$ is also perspective to $[0,a]$ in $L$ and $[u,v]$, as a lattice, is isomorphic to $[0,a]$ by the isomorphism theorem of intervals in modular lattices, $a$ is an atom of $A$ but not an atom of $L$. 
Thus we can pick an element $b\in L\setminus A$ such that $0<b<a$. 

Let $\alg B=\gen {\set{a,b}}=\latgen{\orb a\cup\orb b}$; we have $A\subsetneq B$. Since $\alg B$ is finite by Lemma~\ref{locfinlemma}, we can pick an atom $d$ of $B$ such that $d\leq b< a$. If $0<i<n$, then 
$a\wedge_L g^i(a)=a\wedge_A g^i(a) =0_A=0_L$
implies $d\wedge g^i(d)\leq a\wedge g^i(a)=0$, and we conclude that $i\neq|\orb d|$. Hence, the choice of $n$ yields $|\orb d|=n$. 
The subalgebra $\alg D=\gen d$ of $\alg L$ is the $n$-dimensional rotational cube by Lemma~\ref{eoJsHGs}, and it is a spanning subalgebra of $\alg L$ by Corollary~\ref{spancor}. Also, $\alg D$ is a spanning subalgebra of $\alg B$ since $d\in B$. 

Now, both $\alg A$ and $\alg D$ are spanning  $n$-dimensional rotational cubes in  $\alg B$. 
Since $d$ is an atom of $B$, it is an atom of $ D$. Hence, $\set{g^i(d):0\leq i<n}$ is the set of all atoms of $D$. Clearly, $\set{g^i(a):0\leq i<n}$ is the set of atoms of $A$.
Since the relation $d<a$ is preserved by $g^i$, we have $h_B(g^i(d))<h_B(g^i(a))$, where $h_B$ denotes the height function of $B$. Hence, applying Lemma~\ref{bolaheigts} first to $D$ and $B$, and later to $A$ and $B$, we obtain
\[\length B=\sum_{0<i\leq n}h_B(g^i(d)) < 
\sum_{0<i\leq n}h_B(g^i(a))= \length B,
\]
which is a contradiction.
\end{proof}

\begin{lemma}\label{jdkcGs}  Let $\alg L$ be a subdirectly irreducible distributive rotational lattice of order $n$.  Then $\alg L$  is $($isomorphic to$)$ the $n$-dimensional rotational cube $\algboole n$.
\end{lemma}

\begin{proof} Let $m=\max\set{|\orb a|: a\in L}$. 
By  Lemma~\ref{odCkGBm}, $\alg L\cong \algboole m$. Since $\algboole m$ is of order $m$, we obtain $m=n$. Thus $\alg L\cong \algboole n$.
\end{proof}

\begin{lemma}\label{lemKorl}
Let $I$ be a non-empty subset of $\,\mathbb N$. For each $i\in I$, let $\alg L_i$ be a rotational lattice of order $i$. If $I$ is finite, then the direct product $\prod_{i\in I} \alg L_i$ is a rotational lattice whose order is the least common multiple of $I$. If $I$ is infinite, then $\prod_{i\in I}\alg L_i$ is not a rotational lattice.
\end{lemma}

\begin{proof} Let $\alg L= \prod_{i\in I} \alg L_i$. If $g^t=\id$ holds in $\alg L$, then it holds in $\alg L_i$ since this property is inherited by homomorphic images. On the other hand, for any rotational lattice $\alg M$, $g^t=\id$ holds in $\alg M$ if{f} the order of $\alg M$ divides $t$. 

Now assume that $I$ is finite, and let $m$ denote the least common multiple of $I$. Clearly, $g^m=\id$ holds in $\alg L$. Furthermore, if $g^t=\id$ holds in $\alg L$, then it holds in all $\alg L_i$, which implies that $i$ divides $t$. This yields that $m$ is the order of $\alg L$.

Finally, to obtain a contradiction, assume that $\alg L$ is a rotational lattice, albeit $I$ is infinite. Let $n$ be the order of $\alg L$, and pick an $i\in I$ such that $n<i$. Then $g^n=\id$ holds in $\alg L$ and also in $\alg L_i$, which contradicts the fact that $\alg L_i$ is of order $i$.
\end{proof}

\begin{lemma}\label{simplebn}
For every $n\in\mathbb N$, $\algboole n$ is simple.
\end{lemma}

\begin{proof}
Let $\Theta$ be a congruence of $\algboole n$, distinct from the least congruence. Then there are $a\prec b$ such that $\pair ab\in\Theta$. Let $c$ be the (unique) relative complement of $a$ in $[0,b]$. It is an atom, say $\atom nj$.
Clearly, $\pair 0c \in \Theta$. Hence, $\pair {0}{g^{i+j}(\atom n0)}= \pair {g^i(0)}{g^i(c)}\in \Theta$ for $i=0,\ldots, n-1$. Thus 
$\pair 01=\pair {\bigvee_i 0}{\bigvee_i g^{i+j}(\atom n0)}\in \Theta$, which implies $\Theta =L^2$. This shows that $\algboole n$ is a simple algebra.
\end{proof}

\begin{lemma}\label{lmhimifdif} For $m,n\in \mathbb N$, $\algboole m$ is a homomorphic image of a subalgebra of $\algboole n$ if{f} $m$ divides $n$.
\end{lemma}

\begin{proof} Assume that $m$ divides $n$, and let $k=n/m$. As previously, $\atom n0$, \dots, $\atom n{n-1}$ are the atoms of $\algboole n$, and $g(\atom ni)=\atom n{i+1}$, calculating the subscripts modulo $n$. For $j=0,\ldots,m-1$, let $\btom mj=\atom nj\vee\atom n{m+j} \vee\atom n{2m+j} \vee\dots\vee \atom n{(k-1)m+j}$. These $\btom mj$ are obviously independent in von Neumann's sense, see Gr\"atzer~\cite[V.1.6]{GGLT}, thus they generate a boolean sublattice of length $m$. Since $g(\btom mj)= \btom m{j+1}$, this sublattice is isomorphic to $\algboole m$. That is, $\algboole m$ is a homomorphic image of (actually, isomorphic to) a sublattice of $\algboole n$.

In order to prove the converse, assume that $\algboole m$ is a  homomorphic image of a subalgebra $\alg A$ of $\algboole n$. We can also assume that $m>1$ since otherwise the desired divisibility, $m\mid n$, trivially holds. Since $g\neq\id$ in $\algboole m$, we conclude that $g\neq\id$ in $\alg A$. Hence, by Corollary~\ref{spancor}, $\alg A$ is a spanning subalgebra of $\algboole n$. Let $b$ be an atom of $\alg A$. Note that $b\neq 1_A$ since $m>1$. Applying Lemma~\ref{stablemma} to $\algboole n$, we obtain that $b$ is not stable. Let $t=|\orb b|$. For $i=0,\ldots, t-1$, the set $\bigset{j\in\set{0,\dots,n-1}: \atom nj\leq g^i(b)}$ is denoted by $J_i$.
Note that $=g^i(b)=\bigvee\set{\atom nj: j\in J_i}$.
For $i\neq j$ and $i,j\in\set{0,\ldots,t-1}$, we have $g^i(b)\wedge g^j(b)=0_{\algboole n}$ since $0_A=0_{\algboole n}$ by 
Corollary~\ref{spancor} and since $g^i(b)$ and $g^j(b)$ are distinct atoms of $\alg A$.  Therefore the sets $J_0,\dots, J_{t-1}$ are pairwise disjoint. Since $g$ preserves height, each of the $g^i(b)$ has the same height  in $\algboole n$, and thus we have $|J_0|=\dots =|J_{t-1}|$. 
To show that $J_0\cup\dots\cup J_{t-1}$ equals $\set{0,\dots, n-1}$, let $i\in \set{0,\dots, n-1}$. Pick a $j\in J_0$. We have $\atom nj\leq g^0(b)=b$. By the definition of $\algboole n$, there is a $k\in\set{0,\ldots,n-1}$ such that $\atom ni=g^k(\atom nj)\leq g^k(b)\in\set{g^0(b),\dots,g^{t-1}(b)}$. Hence $i\in J_0\cup\dots\cup J_{t-1}$, and $J_0\cup\dots\cup J_{t-1}$ equals $\set{0,\dots, n-1}$.   
Now, we are in the position to conclude $n=t\cdot|J_0|$, which yields that  $t$  divides $n$.

Next, let $\alg D=\gen b=\latgen{\set{g^i(b):0\leq i<t}}$. Clearly,  $D\subseteq A$. Lemma~\ref{atomBool} implies  $\alg D\cong \algboole t$. To prove $A=D$, let $x\in A$,  denote $\bigset{j\in\set{0,\ldots,n-1}: \atom nj\leq x}$ by $J$, and let  $i\in\set{0,\dots,t-1}$.
If  $g^i(b)\leq x$, then $J_i\subseteq J$. Otherwise, $g^i(b)\wedge x=0_A=0_L$ since  
$g^i(b)$ is an atom of $A$, and we have $J_i\cap J=\emptyset$.  Thus $J$ is the union of some of the $J_i$, $x$ is the join of some of the $g^i(b)$, and we obtain $x\in D$.

Finally, $\alg A=\alg D \cong \algboole t$ is a simple algebra by Lemma~\ref{simplebn}. Since its homomorphic image, $\algboole m$, is not a singleton, we conclude  $\algboole m\cong\algboole t$. This implies  $m=t$, and thus $m$ divides $n$.
\end{proof}

\begin{lemma}\label{sirrg} For $X\in\oid$ and a subdirectly irreducible rotational lattice $\alg L$, we have $\alg L\in\vvv X$ if{f} $\alg L\cong \algboole n$ for some $n\in X$. 
\end{lemma}

\begin{proof} The ``if'' part is trivial by the definition of $\vvv X$. To prove the converse implication, assume $\alg L\in\vvv X$.
Let $n$ denote the order of $\alg L$.
By Lemma~\ref{jdkcGs}, we can assume that $\alg L=\algboole n$.
Since rotational lattices have lattice reducts, they are congruence distributive. 
We obtain from J\'onsson~\cite{jonsson}, see also Burris and Sankappanavar~\cite[Corollary 6.10]{burrissankap}, that $\alg L=\algboole n$ is a homomorphic image of a subalgebra of $\algboole m$ for some $m\in X$. Thus Lemma~\ref{lmhimifdif} yields that $n$ divides $m$. Hence $n\in X$.
\end{proof}

\begin{proof}[Proof of Theorem~\ref{thmmain}]
Part \eqref{thmmaina} follows from Lemmas~\ref{jdkcGs} and \ref{simplebn}. 

Next, to prove part \eqref{thmmainb}, assume that $\var W$ is a variety of rotational lattices. By Lemma~\ref{lemKorl}, 
$\set{n\in \mathbb N: \text{there is an }\alg L\text{ in }\var W\text{ with order } n }$ is a finite set. This fact, combined with Lemma~\ref{jdkcGs}, yields that there is a finite subset $X$ of $\mathbb N$ such that, up to isomorphism, $\set{\algboole n: n\in X}$ is the set of  subdirectly irreducible algebras of $\var W$. (Note that $X=\emptyset$ if{f} $\var W$ is the trivial variety consisting of singleton algebras; the theorem trivially holds for this particular case.) 
We know that $\var W$ is closed under taking subalgebras and homomorphic images. Hence, if $n\in X$, $m\in \mathbb N$, and $m$ divides $n$, then $m\in X$ by Lemma~\ref{lmhimifdif}. This shows $X\in \oid$. 
Hence, by Lemma~\ref{sirrg}, $\var W$ and $\vvv X$ have exactly the same subdirectly irreducible algebras. This implies $\var W=\vvv X$.

Finally, part  \eqref{thmmainc}  is a trivial consequence of Lemma~\ref{lmhimifdif} and part  \eqref{thmmainb}.
\end{proof}

\end{document}